\documentclass[a4paper,12pt]{article}
    \usepackage[top=2cm,bottom=2cm,left=2.5cm,right=2.5cm]{geometry}
    \usepackage{cite, amsmath, amssymb}
    \usepackage[margin=1cm,%
                font=small,%
                format=hang,%
                labelsep=period,%
                labelfont=bf]{caption}
\usepackage[latin1]{inputenc}
\usepackage{amsmath}
\usepackage{amsfonts}
\usepackage{amssymb}
\usepackage{cite}
\usepackage{amsthm}
\usepackage{makeidx}
\usepackage{graphicx}
\usepackage{mathrsfs,xcolor}
\usepackage{enumerate}
\usepackage{blkarray}
\usepackage{bm}
\usepackage{setspace}
\usepackage{float}
\usepackage{authblk}
\usepackage{multirow}
\usepackage{booktabs}
\usepackage[ruled,vlined]{algorithm2e}
\usepackage{blkarray}
\usepackage{authblk}
\usepackage{multirow}

\numberwithin{table}{section}
\numberwithin{equation}{section}

\theoremstyle{plain}
\newtheorem{theorem}{Theorem}[section]
\newtheorem{proposition}[theorem]{Proposition}
\newtheorem{definition}[theorem]{Definition}

\newtheorem{example}[theorem]{Example}

\newtheorem{corollary}[theorem]{Corollary}
\newtheorem{remark}[theorem]{Remark}

\usepackage{authblk}
\author[1,2]{ \textbf{Bryan S. Hernandez}}
\author[3,4,5,6]{\textbf{Eduardo R. Mendoza}}

\affil[1]{\small \textit{Institute of Mathematics, University of the Philippines Diliman, Quezon City 1101, Philippines}}
\affil[2]{\small \textit{Biomedical Mathematics Group, Pioneer Research Center for Mathematical and Computational Sciences, Institute for Basic Science, Daejeon 34126, Republic of Korea}}
\affil[3]{\small \textit{Mathematics and Statistics Department, De La Salle University, Manila  0922, Philippines}}
\affil[4]{\small \textit{Center for Natural Sciences and Environmental Research, De La Salle University, Manila 0922, Philippines}}
\affil[5]{\small \textit{Max Planck Institute of Biochemistry, Martinsried, Munich, Germany}}
\affil[6]{\small \textit{LMU Faculty of Physics, Geschwister -Scholl- Platz 1, 80539 Munich, Germany}}
\affil[*]{Email addresses: \texttt{bryan.hernandez@upd.edu.ph} and \texttt{eduardo.mendoza@dlsu.edu.ph}}

\title{\textbf{Weakly Reversible CF-Decompositions of Chemical Kinetic Systems}}
\date{}

\begin{document}
\maketitle
\begin{abstract} 
This paper studies chemical kinetic systems which decompose into weakly reversible complex factorizable (CF) systems. Among power law kinetic systems, CF systems (denoted as PL-RDK systems) are those where branching reactions of a reactant complex have identical rows in the kinetic order matrix. Mass action and generalized mass action systems (GMAS) are well-known examples. Schmitz's global carbon cycle model is a previously studied non-complex factorizable (NF) power law system (denoted as PL-NDK). We derive novel conditions for the existence of weakly reversible CF-decompositions and present an algorithm for verifying these conditions. We discuss methods for identifying independent decompositions, i.e., those where the stoichiometric subspaces of the subnetworks form a direct sum, as such decompositions relate positive equilibria sets of the subnetworks to that of the whole network. We then use the results to determine the positive equilibria sets of PL-NDK systems which admit an independent weakly reversible decomposition into PL-RDK systems of PLP type, i.e., the positive equilibria are log-parametrized, which is a broad generalization of a Deficiency Zero Theorem of Fortun et al. (2019).\\ \\
	{\bf{Keywords:}} {chemical kinetic systems, chemical reaction networks, non-complex factorizable systems, complex factorizable systems, weakly reversible decompositions, independent decompositions, power law, carbon cycle model}
	
\end{abstract}

\thispagestyle{empty}
\section{Introduction}
\label{sec:intro}

Studies of chemical reaction networks and kinetic systems have hitherto focused on complex factorizable (CF) systems, in particular, on mass action systems and the ``generalized mass action systems'' (GMAS) of M\"uller and Regensburger \cite{MURE2014}. CF systems are defined by the property that at each reactant complex, all branching reactions have the same interaction function. In a mass action system, the interaction function is determined by the reactant's stoichiometric coefficients while in a GMAS, it is given by the kinetic complex, i.e., the image of the kinetic map. CF systems, where underlying network with $\mathscr{S}$, $\mathscr{C}$, and $\mathscr{R}$ as sets of species, complexes, and reactions, are precisely those which have a factorization of its species formation rate function (i.e., its vector field) 
$f\left( x \right) = Y{A_k}{\Psi _K}(x)$
where $Y: \mathbb{R}^\mathscr{C} \to \mathbb{R}^\mathscr{S}$ is map given by the matrix of complexes, $A_k: \mathbb{R}^\mathscr{C} \to \mathbb{R}^\mathscr{C}$  is a Laplacian, and $\Psi _K: \Omega \to \mathbb{R}^\mathscr{C}$ is the kinetics' ``factor map'' on its definition domain in $\mathbb{R}^\mathscr{C}_{\ge 0}$. The subset of CF power law systems is denoted by PL-RDK (power law with reactant-determined kinetics) and can be viewed as a subset of GMAS \cite{TAM2018}.

Much less attention has been paid to the complementary set of non-complex factorizable (NF) systems, though Arceo et al. \cite{arceo2015,AJLM2017} pointed out that representations as chemical reaction networks with power law kinetics of real-world biochemical systems were NF systems. They also identified two classes of kinetic systems, which frequently had NF members, span surjective kinetics \cite{AJLM2017} and RKS (reactant-determined kinetic subspace) kinetics \cite{AJLM2018}. Furthermore, Fortun et al. \cite{fortun2} showed that the power law  kinetic system of Schmitz's model of the earth's pre-industrial carbon cycle model was a weakly reversible, deficiency zero NF system. They also discovered that this PL-NDK system (PL-NDK system = power law NF system) has a decomposition into weakly reversible PL-RDK systems and used this relationship to derive a Deficiency Zero Theorem for a class of NF systems.

A general approach to the study of NF systems was provided in Nazareno et al. \cite{NEML2019} through the family of CF-RM (Complex Factorization by Reactant Multiples) transformations. These techniques partitioned each reactant's set of branching reactions into CF-subsets, i.e., subsets with the same interaction function, and added a minimum number of new reactants to construct a dynamically equivalent CF system with the same number of reactions, identical stoichiometric subspace and the same kinetics. In the paper, the CF-RM transformations were used to derive a theorem for the coincidence of the kinetic and stoichiometric subspaces for NF systems as well as a general computational solution to the linear conjugacy problem for chemical kinetic systems. Hernandez et al. \cite{HMR2020a,HMR2020b} used the CF-RM techniques to provide a computational approach to multistationarity in power law kinetic systems and elucidate the connections to the fundamental decompositions of the underlying network. (Recall that a decomposition, as introduced by M. Feinberg in 1987,  express a network as the union of subnetworks defined by a partition of the network's reaction set. A review of decomposition theory is provided in Section \ref{sec:decomposition}). Magpantay et al. \cite{MHRM2020} combined the previous methods with the STAR-MSC transformation to address multistationarity in poly-PL systems, i.e., sums of power law systems in both CF and NF systems.

A weakness of the CF-RM approach, which is primarily dynamic invariance-oriented, is that the addition of new reactant complexes and (possibly) product complexes leads to considerable changes in connectivity, e.g., in the incidence matrix. In particular, the important property of weak reversibility is very rarely preserved. Given that most results related to equilibria of chemical systems, even in the CF case, assume weak reversibility, this non-invariance phenomenon seriously restricts the utility of the CF-RM techniques.

In this paper, we build on the work of Fortun et al. \cite{fortun2} and develop a general approach for the analysis of weakly reversible NF systems with a weakly reversible decomposition into CF subsystems. This approach is a good complement to the CF-RM techniques for this subclass of NF systems. As our main results in this paper, we provide
\begin{enumerate}
	\item a characterization of weakly reversible NF systems with a weakly reversible CF-decomposition, which also provides necessary conditions for the existence of independent or incidence independent weakly reversible CF-decompositions for the NF system,
	\item an algorithm for determining the weakly reversible CF-decompositions of an NF system (if they exist),
	\item method for determining the independent, weakly reversible CF-decompositions of the NF system (if they exist),
	\item the construction of the CRN of kinetic complexes of a weakly reversible power law kinetic system as the framework for analyzing subnetworks induced by weakly reversible PL- RDK decompositions,
	\item an extension of the ``Deficiency Zero Theorem'' of Fortun et al., i.e., the equilibria existence and parametrization statements, to weakly reversible PL-NDK systems with special PL-RDK decompositions and positive deficiency.
\end{enumerate}
The last result emphasizes that the main structural property in the result of Fortun et al. ensuring the existence and parametrization of equilibria is the availability of the special PL-RDK decomposition rather than the zero deficiency itself. We use the PL-NDK system of Schmitz's carbon cycle model to illustrate our results.

Our paper is organized as follows: Section \ref{sec:prelim} collects the basic concepts and results on chemical reaction networks and kinetic systems needed in the later sections. In Section \ref{sec:decomposition}, a brief review of decomposition theory is provided and Schmitz's global carbon cycle model is introduced. Section \ref{sec:wr:decom} discusses the necessary and sufficient condition that characterizes the existence of weakly reversible CF-decompositions for an NF system. In Section \ref{sec:algo:wr}, an algorithm for determining weakly reversible CF-decompositions (should they exist) is presented. Methods for determining the independent decompositions among the weakly reversible CF-decompositions are discussed in Section \ref{sec:construct:independent}. In Section \ref{sec:existence:parametrization}, the reaction network of kinetic complexes of a weakly reversible power law system is used as a framework for identifying subnetworks induced by its weakly reversible CF-decompositions. This discussion also provides the setting for the extension of the result of Fortun et al.  Summary, conclusions and outlook comprise Section \ref{sec:sum}.

\section{Fundamentals of chemical reaction networks and kinetic systems}
\label{sec:prelim}

\indent We provide important concepts of chemical reaction networks and chemical kinetic systems, which are essential for this paper \cite{arceo2015,feinberg,feinberg1}.

\subsection{Fundamentals of chemical reaction networks}

\begin{definition}
	A {\bf chemical reaction network} (CRN) $\mathscr{N}$ is a triple $\left(\mathscr{S},\mathscr{C},\mathscr{R}\right)$ of nonempty finite sets $\mathscr{S}$, $\mathscr{C} \subseteq \mathbb{R}_{\ge 0}^\mathscr{S}$, and $\mathscr{R} \subset \mathscr{C} \times \mathscr{C}$, of $m$ species, $n$ complexes, and $r$ reactions, respectively, satisfying the following properties: (i.) $\left( {{C_i},{C_i}} \right) \notin \mathscr{R}$ for each $C_i \in \mathscr{C}$, and
	(ii.) for each $C_i \in \mathscr{C}$, there exists $C_j \in \mathscr{C}$ such that $\left( {{C_i},{C_j}} \right) \in \mathscr{R}$ or $\left( {{C_j},{C_i}} \right) \in \mathscr{R}$.
	
\end{definition}

\begin{definition}
	The {\bf molecularity matrix} $Y$ is an $m\times n$ matrix such that $Y_{ij}$ is the stoichiometric coefficient of species $X_i$ in complex $C_j$.
	The {\bf incidence matrix} $I_a$ is an $n\times r$ matrix where 
	$${\left( {{I_a}} \right)_{ij}} = \left\{ \begin{array}{rl}
		- 1&{\rm{ if \ }}{C_i}{\rm{ \ is \ in \ the\ reactant \ complex \ of \ reaction \ }}{R_j},\\
		1&{\rm{  if \ }}{C_i}{\rm{ \ is \ in \ the\ product \ complex \ of \ reaction \ }}{R_j},\\
		0&{\rm{    otherwise}}.
	\end{array} \right.$$
	The {\bf stoichiometric matrix} $N$ is the $m\times r$ matrix given by 
	$N=YI_a$.
\end{definition}

\begin{definition}
	The {\bf reaction vectors} for a given reaction network $\left(\mathscr{S},\mathscr{C},\mathscr{R}\right)$ are the elements of the set $\left\{{C_j} - {C_i} \in \mathbb{R}^m|\left( {{C_i},{C_j}} \right) \in \mathscr{R}\right\}.$
\end{definition}

\begin{definition}
	The {\bf stoichiometric subspace} of a reaction network $\left(\mathscr{S},\mathscr{C},\mathscr{R}\right)$, is given by $$S = {\rm{span}}\left\{ {{C_j} - {C_i} \in \mathbb{R}^m|\left( {{C_i},{C_j}} \right) \in \mathscr{R}} \right\}.$$ The {\bf rank} of the network is given by $s=\dim S$. The set $\left( {x + S} \right) \cap \mathbb{R}_{ \ge 0}^m$ is said to be a {\bf stoichiometric compatibility class} of $x \in \mathbb{R}_{ \ge 0}^m$.
	Two vectors $x, x^{*} \in {\mathbb{R}^m}$ are {\bf stoichiometrically compatible} if $x-x^{*}$ is an element of the stoichiometric subspace $S$.
\end{definition}


\begin{definition}
	Let $\mathscr{N} = \left(\mathscr{S},\mathscr{C},\mathscr{R}\right)$ be a CRN. The {\bf{map of complexes}} $Y:{\mathbb{R}^n} \to \mathbb{R}_{ \ge 0}^m$
	maps the basis vector ${\omega _y}$ to the complex $y \in \mathscr{C}$. The {\bf{incidence map}} ${I_a}:{\mathbb{R}^r} \to {\mathbb{R}^n}$
	is defined by mapping for each reaction $R_i : y \to y' \in \mathscr{R}$, the basis vector $\omega_i$ to the vector ${\omega _{y'}} - {\omega _y} \in \mathscr{C}$. The {\bf{stoichiometric map}} $N : \mathbb{R}^r \to \mathbb{R}^m$ is defined as $N=Y \circ I_a$.
\end{definition}

The {\bf{reactant map}} $\rho: \mathscr{R} \to \mathscr{C}$ associates a reaction $y \to y'$ to its reactant $y$. Denote by $n_r$ the number of reactant complexes. The map is surjective if and only if $n_r=n$. In this case, the network is {\bf cycle terminal}. It is injective if and only if $n_r=r$. In this case, the network is {\bf non-branching}.
The reactant map $\rho$ induces a further useful mapping called {\bf{reactions map}}. The
reactions map $\rho'$ associates the coordinate of reactant complexes to the coordinates of all
reactions of which the complex is a reactant, i.e.,
$\rho' : \mathbb{R}^n \to \mathbb{R}^r$ such that $f : \mathscr{C} \to \mathbb{R}$ mapped to $f \circ \rho$.

We associate the following important linear maps to a positive element $k \in \mathbb{R}^r$: the {\bf{$k$-diagonal map}} ${\rm{diag}}(k)$ that maps $\omega_i$ to $k_i\omega_i$ for $i \in \mathscr{R}$, the {\bf{$k$-incidence map}} $I_k$, which is defined as the composition ${\rm{diag}}(k) \circ \rho'$, and the {\bf{Laplacian map}} $A_k : \mathbb{R}^n \to \mathbb{R}^n$, which is given by $A_k = I_a \circ I_k$.

From the classic graph theory, a {\bf{directed graph}} or simply {\bf{digraph}}, denoted by $D$, consists of a non-empty finite set of vertices and a finite set ordered pairs called arcs. For an arc $(u, v)$, the first vertex $u$ is its tail and the second vertex $v$ is its head \cite{BG2009}.

\begin{definition}
	Let $D$ be a digraph. A {\bf{walk}} $W$ in $D$ is an alternating sequence of vertices $x_i$ and arcs $a_j$ from $D$ given by $$x_1a_1x_2a_2x_3,\ldots, x_{k-1}a_{k-1}x_k$$
	such that the tail of $a_i$ is $x_i$ and the head
	of $a_i$ is $x_{i+1}$ for each $i =1,2,\ldots, k - 1$. A walk is closed if $x_1 = x_k$. Otherwise, it is open. 
\end{definition}

A trail is a walk such that all arcs are distinct. If the vertices of walk $W$ are distinct, $W$ is a called {\bf{path}}. If the vertices $x_1, x_2, \ldots, x_{k-1}$ are distinct with $k \ge 3$ and $x_1 = x_k$, the walk is called a {\bf{cycle}} \cite{BG2009}.

For this paper, we consider ``cycles'' to be simple cycles, i.e., no repeated vertices and no repeated arcs. 

Chemical reaction networks are digraphs where complexes are vertices and reactions are arcs. If there is a path between two vertices $C_i$ and $C_j$, then they are said to be {\bf connected}. If there is a directed path from vertex $C_i$ to vertex $C_j$ and vice versa, then they are said to be {\bf strongly connected}. If any two vertices of a subgraph are {\bf (strongly) connected}, then the subgraph is said to be a {\bf (strongly) connected component}. The (strong) connected components are precisely the {\bf (strong) linkage classes} of a chemical reaction network. The maximal strongly connected subgraphs where there are no arcs from a complex in the subgraph to a complex outside the subgraph is said to be the {\bf terminal strong linkage classes}.
We denote the number of linkage classes, the number of strong linkage classes, and the number of terminal strong linkage classes by $l, sl,$ and $t$, respectively.
A chemical reaction network is said to be {\bf weakly reversible} if $sl=l$, i.e., each linkage class is strongly connected, and it is said to be {\bf $t$-minimal} if $t=l$.

\begin{definition}
	The {\bf deficiency} of a CRN is $\delta=n-l-s$ where $n$ is the number of complexes, $l$ is the number of linkage classes, and $s$ is the rank of the network.
\end{definition}

\subsection{Fundamentals of chemical kinetic systems}

\begin{definition}
	A {\bf kinetics} $K$ for a reaction network $\mathscr{N}$ is an assignment to each reaction $j: y \to y' \in \mathscr{R}$ of a rate function ${K_j}:{\Omega _K} \to {\mathbb{R}_{ \ge 0}}$ such that $\mathbb{R}_{ > 0}^m \subseteq {\Omega _K} \subseteq \mathbb{R}_{ \ge 0}^m$, $c \wedge d \in {\Omega _K}$ if $c,d \in {\Omega _K}$, and ${K_j}\left( c \right) \ge 0$ for each $c \in {\Omega _K}$.
	Furthermore, it satisfies the positivity property: supp $y$ $\subset$ supp $c$ if and only if $K_j(c)>0$.
	The system $\left(\mathscr{N},K\right)$ is called a {\bf chemical kinetic system}.
\end{definition}

\begin{definition}
	The {\bf species formation rate function} (SFRF) of a chemical kinetic system is given by $f\left( x \right) = NK(x)= \displaystyle \sum\limits_{{C_i} \to {C_j} \in \mathscr{R}} {{K_{{C_i} \to {C_j}}}\left( x \right)\left( {{C_j} - {C_i}} \right)}.$
\end{definition}
The ODE or dynamical system of a chemical kinetics system is $\dfrac{{dx}}{{dt}} = f\left( x \right)$. An {\bf equilibrium} or {\bf steady state} is a zero of $f$.

\begin{definition}
	The {\bf set of positive equilibria} of a chemical kinetic system $\left(\mathscr{N},K\right)$ is given by ${E_ + }\left(\mathscr{N},K\right)= \left\{ {x \in \mathbb{R}^m_{>0}|f\left( x \right) = 0} \right\}.$
\end{definition}

A chemical reaction network is said to admit {\bf multiple equilibria} if there exist positive rate constants such that the ODE system admits more than one stoichiometrically compatible equilibria. On the other hand, the {\bf set of complex balanced equilibria} \cite{HornJackson} is given by 
\[{Z_ + }\left(\mathscr{N},K\right) = \left\{ {x \in \mathbb{R}_{ > 0}^m|{I_a} \cdot K\left( x \right) = 0} \right\} \subseteq {E_ + }\left(\mathscr{N},K\right).\]
A positive vector $c \in \mathbb{R}^m$ is complex balanced if $K\left( c \right)$ is contained in $\ker{I_a}$. A chemical kinetic system is complex balanced if it has a complex balanced equilibrium.

\begin{definition}
	A kinetics $K$ is a {\bf power law kinetics} (PLK) if 
	${K_i}\left( x \right) = {k_i}{{x^{{F_{i}}}}} $  $\forall i =1,...,r$ where ${k_i} \in {\mathbb{R}_{ > 0}}$ and ${F_{ij}} \in {\mathbb{R}}$. The power law kinetics is defined by an $r \times m$ matrix $F$, called the {\bf kinetic order matrix} and a vector $k \in \mathbb{R}^r$, called the {\bf rate vector}.
\end{definition}
If the kinetic order values are the corresponding stoichiometric coefficients, then the system has the {\bf mass action kinetics (MAK)}.

\begin{definition}
	A PLK system has {\bf reactant-determined kinetics} (of type PL-RDK) if for any two reactions $i, j$ with identical reactant complexes, the corresponding rows of kinetic orders in $F$ are identical. That is, ${f_{ik}} = {f_{jk}}$ for $k = 1,2,...,m$. A PLK system has {\bf non-reactant-determined kinetics} (of type PL-NDK) if there exist two reactions with the same reactant complexes whose corresponding rows in $F$ are not identical.
\end{definition}

We now consider the {\bf{complex factorizable kinetics (CFK)}} and its complement, the {\bf{non-complex factorizable kinetics (NFK)}} \cite{arceo2015}.

\begin{definition}
	A chemical kinetics $K$ is {\bf{complex factorizable}} (CF) if there is a $k \in \mathbb{R}_{ > 0}^r$ and a mapping ${\Psi _K}:{\Omega _K} \to \mathbb{R}_{ \ge 0}^n$ such that $\mathbb{R}_{ > 0}^m \subseteq {\Omega _K} \subseteq \mathbb{R}_{ \ge 0}^m$ and $K = {I_k} \circ {\Psi _K}$. We call ${\Psi _K}$ a {\bf{factor map}} of $K$.
\end{definition}

We have $K = {\rm{diag}}\left( k \right) \circ \rho ' \circ {\Psi _K}$ via the definition of the $k$-incidence map $I_k$. This implies that a complex factorizable kinetics $K$ has a decomposition into diagonal matrix of
rate constants and an interaction map $\rho ' \circ {\Psi _K}:\mathbb{R}_{ \ge 0}^m \to \mathbb{R}_{ \ge 0}^r$. Note that the values of the interaction map are ``reaction-determined''. It basically means that they are determined by the values on the reactant complexes.
CF kinetics (CFK) generalizes the key structural property of MAK. The complex formation rate function decomposes as $g = {A_k} \circ {\Psi _K}$ while the species formation rate function factorizes as $f = Y \circ {A_k} \circ {\Psi _K}$. The CFK systems in the set of PLK systems are precisely the PL-RDK systems \cite{arceo2015}.


We now formally define the RID (rate constant-interaction map decomposable) kinetics.

\begin{definition}
	A kinetics $K$ is an {\bf{RID kinetics}} if for any $x \in \Omega$ (its domain of definition) given by $$K(x) = {\text{diag}(k)}I_K(x)$$ where $k$ a positive vector in $\mathbb R^\mathscr{R}$ and $I_K(x)$ does not depend on $k$.  
\end{definition}

\section{Decompositions and equilibria of kinetic systems}
\label{sec:decomposition}

In this section, we provide a brief review of decomposition theory with a focus on concepts and results to be used in the following sections. We introduce the power law representation of Schmitz's global carbon cycle model first studied in \cite{fortun2}. We use it and a subnetwork as running examples for our results.

\subsection{A brief review of decomposition theory}
Decomposition theory was initiated by M. Feinberg in 1987 in his review \cite{feinberg12}. He introduced the following definition of a decomposition:

\begin{definition}
	A {\textbf{decomposition}} of a reaction network $\mathscr{N}$ is a set of subnetworks $\{\mathscr{N}_1, \mathscr{N}_2,...,\mathscr{N}_k\}$ of $\mathscr{N}$ induced by a partition $\{\mathscr{R}_1, \mathscr{R}_2,...,\mathscr{R}_k\}$ of its reaction set $\mathscr{R}$. 
\end{definition}

The use of the term ``decomposition'' is not consistent in the CRNT literature as some authors require only that $\mathscr{R} = \cup \mathscr{R}_i$, but not that the subsets are pairwise disjoint. We denote this looser concept as a covering of the network.

\begin{example}
	The most widely used decomposition of a reaction network is the set of linkage classes. Linkage classes have the special property that they not only partition the set of reactions, but also the set of complexes. They are presented in the simpler latter form (e.g., in CRNToolbox \cite{software}) and hence, not consciously treated as decompositions.
\end{example}

\begin{example}
	The ``decomposition'' with a single ``subnetwork'' $\mathscr{N} = \mathscr{N}_1$ shall be referred to as the ``trivial decomposition''.
\end{example}

We denote a decomposition by 
$\mathscr{N} = \mathscr{N}_1 \cup \mathscr{N}_2 \cup ... \cup \mathscr{N}_k$
since $\mathscr{N}$ is a union of the subnetworks in the sense of \cite{ghms2019}. It also follows that, 
$${S} = {S}_1 + {S}_2 + ... + {S}_k$$
for the corresponding stoichiometric subspaces.

M. Feinberg also introduced the important concept of independence of decompositions:

\begin{definition}
	A network decomposition $\mathscr{N} = \mathscr{N}_1 \cup \mathscr{N}_2 \cup ... \cup \mathscr{N}_k$  is {\textbf{independent}} if its stoichiometric subspace is equal to the direct sum of the stoichiometric subspaces of its subnetworks.
\end{definition}

\begin{example}
	The independence of the linkage class decomposition is an essential network property in various theorems for mass action systems, e.g., the Deficiency One Theorem \cite{feinberg1}.
\end{example}

A basic property of an independent decomposition was first recorded in \cite{fortun2}:
\begin{proposition}
	For an independent decomposition, $\delta \le \delta_1 +\delta_2 ... +\delta_k$.
\end{proposition}

Decompositions can have the following useful relation:

\begin{definition}
	Let $\mathscr{D}: \mathscr{N}  = \mathscr{N}_1 \cup \ldots \cup \mathscr{N}_k$
	and
	$\mathscr{D}': \mathscr{N}  = \mathscr{N}_1' \cup \ldots \cup \mathscr{N}_k'$
	be decompositions induced by $\{\mathscr{R}_1,\ldots,\mathscr{R}_k\}$ and $\{\mathscr{R}_1',\ldots,\mathscr{R}_k'\}$, respectively.
	$\mathscr{D}$ is a {\bf refinement} of $\mathscr{D}'$ (and $\mathscr{D}'$ a {\bf{coarsening}} of $\mathscr{D}$) if each $\mathscr{R}_i$ is contained in an $\mathscr{R}_j'$.  One also say that $\mathscr{D}$ is finer than $\mathscr{D}'$ (or $\mathscr{D}'$ coarser than $\mathscr{D}$). 
\end{definition}

Farinas et al. \cite{FAML2020} noted the following relationship:

\begin{proposition}
	If a decomposition is independent, then any coarsening of the decomposition is independent.
	\label{prop:inde:coarse}
\end{proposition}

M. Feinberg in \cite{feinberg12} demonstrated the importance of the independence concept through the following Theorem:

\begin{theorem} (Feinberg Decomposition Theorem \cite{feinberg12})
	\label{feinberg:decom:thm}
	Let $P(\mathscr{R})=\{\mathscr{R}_1, \mathscr{R}_2,...,\mathscr{R}_k\}$ be a partition of a CRN $\mathscr{N}$ and let $K$ be a kinetics on $\mathscr{N}$. If $\mathscr{N} = \mathscr{N}_1 \cup \mathscr{N}_2 \cup ... \cup \mathscr{N}_k$ is the network decomposition of $P(\mathscr{R})$ and $${E_ + }\left(\mathscr{N}_i,{K}_i\right)= \left\{ {x \in \mathbb{R}^\mathscr{S}_{>0}|N_iK_i(x) = 0} \right\}$$ then
	\[{E_ + }\left(\mathscr{N}_1,K_1\right) \cap {E_ + }\left(\mathscr{N}_2,K_2\right) \cap ... \cap {E_ + }\left(\mathscr{N}_k,K_k\right) \subseteq  {E_ + }\left(\mathscr{N},K\right).\]
	If the network decomposition is independent, then equality holds.
\end{theorem}

Note however that independence does not ensure that the intersection is non-empty though there is a positive equilibrium for each subnetwork. Additional properties of the decomposition and the kinetics are necessary to achieve this. To our knowledge, there are only a few such results.

\begin{example}
	For a network $\mathscr{N}$ with an independent linkage class decomposition and mass action kinetics $K$, $E_+(\mathscr{L}_i, K_i) \ne \varnothing$ $\implies$ $E_+(\mathscr{N}, K) \ne \varnothing$ (s. \cite{BORO2013}).
\end{example}

The corresponding concepts and results for complex balanced equilibria were established only recently in \cite{FAML2020}:

\begin{definition}
	A network decomposition $\mathscr{N} = \mathscr{N}_1 \cup \mathscr{N}_2 \cup ... \cup \mathscr{N}_k$ is {\textbf{incidence independent}} if the incidence map $I_a$ of $\mathscr{N}$ is equal to the direct sum of the incidence maps of its subnetworks.
\end{definition}

Clearly, incidence independence is equivalent to  
$$n-\ell = \displaystyle \sum (n_i - l_i).$$
Incidence independence has the following analogous properties to independence:

\begin{proposition} The following statements hold:
	\begin{enumerate}
		\item For an incidence independent decomposition, $\delta \ge \delta_1 +\delta_2 ... +\delta_k$. 
		\item Any coarsening of an incidence independent decomposition is incidence independent.
	\end{enumerate}
	\label{farinas:incidence:independent}
\end{proposition}

An important result in \cite{FAML2020} is the analogue of Feinberg´s Theorem:

\begin{theorem} (Theorem 4 \cite{FAML2020})
	\label{decomposition:thm:2}
	Let $\mathscr{N}=(\mathscr{S},\mathscr{C},\mathscr{R})$ be a a CRN and $\mathscr{N}_i=(\mathscr{S}_i,\mathscr{C}_i,\mathscr{R}_i)$ for $i = 1,2,...,k$ be the subnetworks of a decomposition.
	Let $K$ be any kinetics, and $Z_+(\mathscr{N},K)$ and $Z_+(\mathscr{N}_i,K_i)$ be the set of complex balanced equilibria of $\mathscr{N}$ and $\mathscr{N}_i$, respectively. Then
	\begin{itemize}
		\item[i.] ${Z_ + }\left(\mathscr{N}_1,K_1\right) \cap {Z_ + }\left(\mathscr{N}_2,K_2\right) \cap ... \cap {Z_ + }\left(\mathscr{N}_k,K_k\right) \subseteq  {Z_ + }\left(\mathscr{N},K\right)$.\\
		If the decomposition is incidence independent, then
		\item[ii.] ${Z_ + }\left( {\mathscr{N},K} \right) = {Z_ + }\left(\mathscr{N}_1,K_1\right) \cap {Z_ + }\left(\mathscr{N}_2,K_2\right) \cap ... \cap {Z_ + }\left(\mathscr{N}_k,K_k\right)$, and
		\item[iii.] ${Z_ + }\left( {\mathscr{N},K} \right) \ne \varnothing$ implies ${Z_ + }\left( {\mathscr{N}_i,K_i} \right) \ne \varnothing$ for each $i=1,...,k$.
	\end{itemize}
\end{theorem}

In contrast to the case of positive equilibria, there is a class of decompositions ensuring the converse of statement iii. To state this, we introduce the concept of the set $\mathscr{C}_\mathscr{D}$ of common complexes of a decomposition, first introduced in \cite{FOFM2021}:

\begin{definition}
	Let $\mathscr{C}_i$  be the set of complexes of a subnetwork $\mathscr{N}_i$ in a decomposition $\mathscr{D}: \mathscr{N}  = \mathscr{N}_1 \cup \ldots \cup \mathscr{N}_k$. The set of common complexes given by $\mathscr{C}_\mathscr{D} := \displaystyle \cup (\mathscr{C}_i \cap \mathscr{C}_j)$, where $i \ne j$ and $i,j = 1,\ldots, k$.
	A decomposition is a $\mathscr{C}^*$-decomposition if $|\mathscr{C}_\mathscr{D}| \le 1$ and a $\mathscr{C}$-decomposition if $|\mathscr{C}_\mathscr{D}| =0$.
\end{definition}

The $\mathscr{C}$-decompositions are those which also partition the set of complexes. In \cite{FAML2020}, it is shown that any \cite{FAML2020} decomposition is a coarsening of the linkage class decomposition. More importantly, the following converse to Statement iii. holds:

\begin{theorem} (Theorem 5 \cite{FAML2020})
	Let $\mathscr{N} = \mathscr{N}_1 \cup \mathscr{N}_2 \cup ... \cup \mathscr{N}_k$ be a weakly reversible $\mathscr{C}$-decomposition of a chemical kinetic system $(\mathscr{N},K)$. If ${Z_ + }\left( {\mathscr{N}_i,K_i} \right) \ne \varnothing$ for each $i=1,...,k$, then ${Z_ + }\left( {\mathscr{N},K} \right) \ne \varnothing$.
\end{theorem}

It is shown in \cite{FAML2020} too, that any $\mathscr{C}^*$-decomposition is incidence independent. As observed in \cite{FOFM2021}, for a decomposition with $|\mathscr{C}_\mathscr{D}| > 1$, this is no longer true in general. For example, the simple network $R_1: X \to Y$, $R_2: Y \to X$ has the decomposition $\{R_1\} \cup \{R_2\}$ with $|\mathscr{C}_\mathscr{D}| =2 $, $n - l = 1 < 1 + 1 = 2$, and hence, not incidence independent. 

The combination of the independence and incidence independence in a decomposition leads to several interesting properties:

\begin{definition}
	A decomposition is {\bf{bi-independent}} if it is both independent and incidence independent.
\end{definition}

\begin{proposition}
	For any decomposition, the following statements are equivalent:
	\begin{enumerate}
		\item[i.] $\delta = \delta_1+ \ldots + \delta_k$ and independent
		\item[ii.] $\delta = \delta_1+ \ldots + \delta_k$ and incidence independent
		\item[iii.] bi-independent
	\end{enumerate}
\end{proposition}

\begin{proof}
	We have $\delta = \delta_1+ \ldots + \delta_k \iff (n-l)-\sum(n_i-l_i)=s-\sum s_i$. If one side is zero, so is the other.
\end{proof}

A zero deficiency decomposition (ZDD) is a decomposition whose subnetworks have zero deficiency. In a deficiency zero network, we obtain the following interesting equivalences:

\begin{proposition}
	If the network has zero deficiency, then the following statements are equivalent:
	\begin{enumerate}
		\item[i.] incidence independent 
		\item[ii.] independent + ZDD
		\item[iii.] bi-independent
	\end{enumerate}
\end{proposition}

\begin{proof}
	(i) $\implies$ (ii) Since deficiency is an upper bound for subnetwork deficiencies in an incidence independent decomposition, it is also ZDD. Since $n - l = s$ and $n_i - l_i = s_i$, independence follows too. (ii) $\implies$ (iii) the same equation read right to left delivers incidence independence, and (iii) $\implies$ (i) is trivial.
\end{proof}



Despite its early founding by M. Feinberg in 1987, decomposition theory has received little attention in the CRNT community. Recently however, its usefulness beyond results on equilibria existence and parametrization, e.g., in the relatively new field of concentration robustness. Interesting results for larger and high deficiency systems have been derived for more general kinetic systems such as power law systems \cite{FOMF2021,FOME2020} and Hill-type systems \cite{HEME2021}.

\subsection{A power law representation of Schmitz's global carbon cycle model}
\label{schmitz:running:example}

\begin{figure}
	\centering
	\begin{minipage}[b]{.7\textwidth}
		\includegraphics[width=\textwidth]{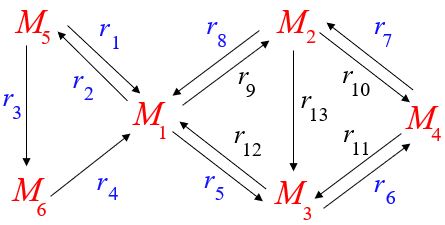}
		\caption{Schmitz's global carbon cycle model adapted from \cite{fortun2}}
		\label{schmitznetwork}
	\end{minipage}\hfill
	
	\begin{minipage}[b]{.45\textwidth}
		\centering
		\label{M}
		\(
		\begin{blockarray}{ccccccl}
			&  &  &  &  &   \\
			{\color{red}M_1} & {\color{red}M_2} & {\color{red}M_3} & {\color{red}M_4} & {\color{red}M_5} & {\color{red}M_6}  \\
			\begin{block}{[cccccc]l}
				{\color{blue}0} & {\color{blue}0} & {\color{blue}0} & {\color{blue}0} & {\color{blue}1} & {\color{blue}0} & {\color{blue}r_1} \\
				{\color{blue}0.36} & {\color{blue}0} & {\color{blue}0} & {\color{blue}0} & {\color{blue}0} & {\color{blue}0} & {\color{blue}r_2} \\
				{\color{blue}0} & {\color{blue}0} & {\color{blue}0} & {\color{blue}0} & {\color{blue}1} & {\color{blue}0} & {\color{blue}r_3} \\
				{\color{blue}0} & {\color{blue}0} & {\color{blue}0} & {\color{blue}0} & {\color{blue}0} & {\color{blue}1} & {\color{blue}r_4} \\
				{\color{blue}1} & {\color{blue}0} & {\color{blue}0} & {\color{blue}0} & {\color{blue}0} & {\color{blue}0} & {\color{blue}r_5} \\
				{\color{blue}0} & {\color{blue}0} & {\color{blue}1} & {\color{blue}0} & {\color{blue}0} & {\color{blue}0} & {\color{blue}r_6} \\
				{\color{blue}0} & {\color{blue}0} & {\color{blue}0} & {\color{blue}1} & {\color{blue}0} & {\color{blue}0} & {\color{blue}r_7} \\
				{\color{blue}0} & {\color{blue}9.4} & {\color{blue}0} & {\color{blue}0} & {\color{blue}0} & {\color{blue}0} & {\color{blue}r_8}\\
				1& 0 & 0 & 0 & 0 & 0 & r_9 \\
				0& 1 & 0 & 0 & 0 & 0 & r_{10} \\
				0& 0 & 0 & 1 & 0 & 0 & r_{11} \\
				0& 0 & 10.2 & 0 & 0 & 0 & r_{12} \\
				0& 1 & 0 & 0 & 0 & 0 & r_{13} \\
			\end{block}
		\end{blockarray}
		\)
		\caption{Kinetic order matrix for Schmitz's global carbon cycle model}\label{kinetic:order:matrix:schmitz}
	\end{minipage}
\end{figure}

We refer to Figure \ref{schmitznetwork}. R. Schmitz's model \cite{schmitz} of the earth's carbon cycle in the pre-industrial state consists of six carbon pools -- $M_1$(atmosphere), $M_2$(warm ocean surface water), $M_3$(cool ocean surface water), $M_4$ (deep ocean waters), $M_5$(terrestrial biota) and $M_6$(soil and detritus) -- and 13 (directed) mass transfers among them. A power law representation of the model was first investigated by Fortun et al. in \cite{fortun2}: the weakly reversible system $(\mathscr{N}, K)$ has zero deficiency since its complexes are monospecies and a PL-NDK system with 3 NF nodes $M_1$, $M_2$ and $M_3$ (s. Figure and kinetic order matrix). Table \ref{table:reactionsets:nodes:schmitz} lists the reaction sets and the CF-subsets of the NF nodes.

\begin{table}
	\centering
	\caption{Reaction sets and CF-subsets of the NF nodes for the Schmitz's global carbon cycle model in Section \ref{schmitz:running:example}}
	\label{table:reactionsets:nodes:schmitz}       
	\begin{tabular}{lll}
		CF node  & Reaction set & CF-subsets \\
		\noalign{\smallskip}\hline\noalign{\smallskip}
		$M_1$ & $\{r_2,r_5,r_9\}$ & $\{r_2\},\{r_5,r_9\}$\\
		$M_2$ & $\{r_8,r_{10},r_{13}\}$ & $\{r_8\}, \{r_{10},r_{13}\}$\\
		$M_3$ & $\{r_6,r_{12}\}$ & $\{r_6\},\{r_{12}\}$\\
		\noalign{\smallskip}\hline\noalign{\smallskip}
	\end{tabular}
\end{table}

The system has a weakly reversible PL-RDK decomposition $\mathscr{N} = \mathscr{N}_1 \cup \mathscr{N}_2 \cup \mathscr{N}_3$, with the subnetworks induced by the following subsets of the reaction set:
\begin{align*}
	\mathscr{R}_1 &= \{r_1, r_2, r_3, r_4\},\\
	\mathscr{R}_2 &= \{r_5, r_6, r_7, r_8\}, {\text{ and}}\\
	\mathscr{R}_3 &= \{r_9, r_{10}, r_{11}, r_{12}, r_{13}\},  {\text{ respectively.}}
\end{align*}
The decomposition is also clearly a zero deficiency decomposition (ZDD). Its set of $\mathscr{C}_\mathscr{D}$ of common complexes = $\{M_1, M_2, M_3, M_4\}$.

The weakly reversible subnetwork $\mathscr{N}':= \mathscr{N}_1 \cup \mathscr{N}_2$ is of special interest since the PL-RDK decomposition for it is also independent. Being also ZDD in a deficiency zero network, the decomposition is bi-independent. It is also a $\mathscr{C}^*$-decomposition. The kinetic order matrix of the subsystem consists of the first 8 rows of the matrix in Figure \ref{kinetic:order:matrix:schmitz}, showing that the subsystem is PL-NDK with one NF-node $M_1$ with 2 CF-subsets $\{r_2\}$ and $\{r_5\}$.

\section{Weakly reversible CF-decompositions of a linkage class}
\label{sec:wr:decom}

Our approach is based on the following Proposition:

\begin{proposition}
	A weakly reversible chemical kinetic system $(\mathscr{N}, K)$ has a weakly reversible CF-decomposition if and only if each of its (weakly reversible) linkage classes has a weakly reversible CF-decomposition.
\end{proposition}

\begin{proof}
	For a weakly reversible CF-decomposition $\mathscr{N} = \mathscr{N}_1 \cup \ldots \cup \mathscr{N}_k$ and any linkage class $\mathscr{L}_j$, the sets $\mathscr{R}(\mathscr{N}_i) \cap \mathscr{R}(\mathscr{L}_j)$ define a weakly reversible CF-decomposition of $\mathscr{L}_j$. The converse is also straightforward. 
\end{proof}

In this section, we investigate conditions regarding existence of weakly reversible CF-decompositions by checking if such decomposition is possible for the linkage classes.

\subsection{Eulerian linkage classes}
\label{eulerian:linkage:classes}
We introduce some concepts that are essential to this paper about digraphs provided in \cite{BG2009,BM2008}.

\begin{definition}
	Let $D$ be a digraph. The {\bf in-degree} of a vertex $v$ in $D$, denoted by $d_-(v)$  is the number of arcs with head $v$, and the {\bf out-degree} of $v$, denoted by $d_+(v)$, is the number of arcs with tail $v$.
\end{definition}

\begin{definition}
	A digraph $D$ is {\bf even} if $d_+(v)=d_-(v)$ for each vertex $v \in V$.
\end{definition}

\begin{definition}
	A digraph $D$ is {\bf Eulerian} if it contains a closed trail (i.e., a walk with no repeated arcs) that contains all of the arcs in $D$. This closed trail is also called an Eulerian Trail.
\end{definition}

Theorem \ref{digraph:cycle:decomposition} was taken from \cite{BG2009}, which was attributed to Veblen, while Theorem \ref{digraph:cycle:decomposition:euler} was taken from the version in \cite{BM2008}, which was called the Euler's Theorem. These results are useful to determine whether a digraph can be decomposed into cycles.

\begin{theorem} 
	A digraph admits a cycle decomposition if and only if it is even.
	\label{digraph:cycle:decomposition}
\end{theorem}

\begin{theorem} 
	A directed multigraph $D$ is Eulerian if and only if $D$ is connected and $d_+(v)=d_-(v)$ for every vertex v of $D$.
	\label{digraph:cycle:decomposition:euler}
\end{theorem}

Since we are studying cycles originating in a reactant complex, it suffices to consider a connected digraph, i.e., a linkage class. We will call a linkage class an NF-linkage class if it contains an NF-reactant complex, otherwise a CF-linkage class.

According to Theorems \ref{digraph:cycle:decomposition} and \ref{digraph:cycle:decomposition:euler}, a digraph is Eulerian if and only if it has a decomposition into directed cycles. In the lucky scenario that a linkage class is Eulerian, we obtain a solution to our problem:

\begin{proposition}
	If an NF linkage class $\mathscr{L}$ of $(\mathscr{N},K)$ is Eulerian, then the cycle decomposition is a weakly reversible CF-decomposition of the NF linkage class.
	\label{eulerian:linkage:classes1}
\end{proposition}

\begin{proof}
	This follows from the simple facts that a cycle is of course weakly reversible and is non-branching, and hence, a CF-decomposition.
\end{proof}


\subsection{Weakly reversible CF coverings of a weakly reversible linkage class with a single NF node}
\label{CF:cycle:covering}

We will define concepts in general for a kinetic system, but clearly the interesting ones that we will focus on are the weakly reversible systems.
Let $y$ be any complex in a weakly reversible linkage class $\mathscr{N}$.  If $K$ is an RIDK kinetics, let $\{{\rm{CF}}_i(y)\}$ be the CF-subsets of $y$, which partition the set $\mathscr{R}(y)$ of $y$'s branching reactions, $i = 1,\ldots, N_R(y)$. Note that $N_R(y) = 1$ iff $y$ is a CF-node. Also, a simple cycle of $y$ contains only 1 branching reaction. 
\begin{definition}
	An $i$-cycle of $y$ is a simple cycle whose branching reaction is in ${\rm{CF}}_i(y)$. Let ${\rm{CFC}}_i(y)$ be the set of all $i$-cycles of $y$. ${\mathscr{CFC}}_i(y)$ is the subnetwork of $\mathscr{N}$ defined by the reactions in $i$-cycles.
	If $y$ is an RDK node, then ${\mathscr{CFC}}_1(y)=\mathscr{N}$.
\end{definition}

\begin{proposition}	
	If $\mathscr{N}$ has a single NF-node $y$, then $\{{\mathscr{CFC}}_i(y)\}$ is a non-trivial weakly reverible CF covering of $\mathscr{N}$, called the CF$(y)$-cycle covering.
\end{proposition}

\begin{proof}
	That it is a covering follows directly from the weak reversibility of $\mathscr{N}$. Each subnetwork is a union of cycles, and hence, weakly reversible. It is also CF by construction (cycles from the single NF node).
\end{proof}
The covering is not necessarily a decomposition because an $i$-cycle and 
a $j$-cycle where $i \ne j$ can have reactions in common. An example of a non-trivial decomposition is given in Example \ref{ex:schmitz}.

\begin{example}
	Consider the Schmitz's subnetwork $\mathscr{N}'$ in Section \ref{schmitz:running:example}, with reactions restricted to the first 8 rows. We consider this example as our running example. These subsets of $\mathscr{R}$: $\{r_1,r_2,r_3,r_4\}$ and $\{r_5,r_6,r_7,r_8\}$ induce the CF$(M_1)$-cycle decomposition of 
	the subnetwork of Schmitz's carbon cycle model.
	\label{ex:schmitz}	
\end{example}

\subsection{A necessary condition for a weakly reversible CF decomposition of a linkage class with a single NF node}
\label{properties:CF:cycle:dec}
We now introduce the following important necessary condition for a weakly reversible CF-decomposition of a linkage class with a single NF node. This result is general whether Eulerian or non-Eulerian. 
\begin{theorem}
	Let $(\mathscr{N},K)$ be an RIDK system with a single linkage class and a single NF-node $y$. If $\mathscr{N} = \mathscr{N}'_1 \cup \ldots \cup \mathscr{N}'_{k'}$ is a weakly reversible CF-decomposition of $\mathscr{N}$, then it is a refinement of a decomposition of the form $\mathscr{D}: \mathscr{N} = \mathscr{N}_1 \cup \ldots \cup \mathscr{N}_k$, with $k = N_R(y)$, $y \in \mathscr{N}_i$ and $\mathscr{R}(\mathscr{N}_i)$ is the set of reactions of $i$-cycles of $y$ contained in $\mathscr{N}_i$.
	\label{main:theorem:1node}
\end{theorem}
\begin{proof}
	Since $\mathscr{N} = \mathscr{N}'_1 \cup \ldots \cup \mathscr{N}'_{k'}$ is a CF-decomposition, in the reaction set of each subnetwork containing $y$, all branching reactions must come from only 1 CF-subset. Furthermore, there may be subnetworks not containing $y$. Hence, $k'\ge N_R(y)$. We coarsen the decomposition, which can clearly be done in several ways, by lumping, i.e., taking the union of all subnetworks with branching reactions from the same CF$_i$ and further lumping subnetworks not containing $y$ with one which contains $y$. We obtain a  weakly reversible CF-decomposition with $N_R(y)$ subnetworks (CF$_i$ contained in its reaction set) such that $y$ is in its set of complexes. Since $\mathscr{N}_i$ is weakly reversible, any reaction is contained in a simple $i$-cycle. Conversely, if an $i$-cycle of $y$ is in $\mathscr{N}_i$, by definition, all of its reactions are reactions in $\mathscr{N}_i$. Thus, if ${\mathscr{CFC}}_i^D(y)$ denotes the subnetwork generated by the reactions of $i$-cycles of $y$ contained in $\mathscr{N}_i$, we have $\mathscr{N}_i = {\mathscr{CFC}}_i^D(y)$.
\end{proof}
The necessary condition also suggests the following procedure for generating a weakly reversible CF-decomposition in a single linkage class and single NF node RIDK system:
\begin{enumerate}
	\item[1.]	Choose a CF-subset of the NF-node $y$, denote it with  CF$_1(y)$. For each branching reaction in CF$_1$, choose a simple cycle of $y$. Let $\mathscr{N}_{0,1}$ be the subnetwork generated by the set $\mathscr{R}_{0,1}$ of reactions of these cycles.
	\item[2.]	If $\mathscr{R} \backslash \mathscr{R}_{0,1} \ne \varnothing$ and the subnetwork generated by $\mathscr{R} \backslash \mathscr{R}_{0,1}$ is not weakly reversible, then the system has no weakly reversible CF-decomposition.
	\item[3.]	Otherwise, for another CF subset, CF$_2(y)$, choose simple cycles which do not contain any reactions from $\mathscr{R}_{0,1}$. If there is a branching reaction in CF$_2(y)$ for which this is not possible, then the system does not have a weakly reversible CF-decomposition. Otherwise, let $\mathscr{N}_{0,2}$ be the subnetwork generated by the set $\mathscr{R}_{0,2}$ of reactions of these cycles.
	\item[4.]	If $\mathscr{R} \backslash (\mathscr{R}_{0,1} \cup \mathscr{R}_{0,2}) \ne \varnothing$ and the subnetwork generated by $\mathscr{R} \backslash (\mathscr{R}_{0,1} \cup \mathscr{R}_{0,2})$ is not weakly reversible, then the system does not have a weakly reversible CF-decomposition.
	\item[5.]	Continue accordingly until all CF subsets for $1,\ldots,N_R(y)$ are covered. If the complement of $\cup \mathscr{R}_{0,i}$ is not empty, find for each reaction a cycle of $y$ with the property that if the cycle's branching reaction belongs to CF$_j$, no reactions from $\cup \mathscr{R}_{0,i}$, $i \ne j$ are contained in the cycle. If there is an instance for which this is not possible, then the system does not have a weakly reversible CF-decomposition. Otherwise, add the cycles to the subnetworks with the same CF$_j$.
\end{enumerate}

The last step of the procedure defines an appropriate weakly reversible CF-decomposition with $N_R(y)$ subnetworks, $y$ in each subnetwork and by construction, the reactions of each subnetwork as those of the $i$-cycles contained in the subnetwork. Because of the many choices made in the procedure, this decomposition is not unique. We provide Algorithm 
\ref{algorithm0} for this procedure.

\begin{algorithm}
	\SetAlgoLined
	\bigskip
	{\bf{STEP 1}} 
	\\
	{\bf{Input 1}} {reaction network $\mathscr{N}$ with reaction set $\mathscr{R}$}\\
	{\bf{Input 2}} {kinetic order matrix $F$}\\
	\bigskip
	{\bf{STEP 2}} 
	\\
	$CF=$ set of CF subsets\\
	$|CF|=$ number of CF subsets\\
	\bigskip
	{\bf{STEP 3}} 
	\\
	$UC0 = \varnothing$\\
	\For{$\beta = 1$ to $|CF|$}{
		\For{$\rho \in CF_{\beta}(y)$}
		{choose a simple cycle $c(y)$ of $y$
			
		}
		$\mathscr{R}_{0,\beta} =$ set of reactions of all these cycles\\
		$UC=\mathscr{R}_{0,\beta} \cup UC0$ set of reactions of all these cycles\\
		\eIf{$\mathscr{R} \backslash UC \ne \varnothing$ and the subnetwork generated by $\mathscr{R} \backslash UC$ is not weakly reversible}{
			the NF system has no weakly reversible CF-decomposition\\
			EXIT the algorithm
		}{
		}
		$UC0=\mathscr{R}_{0,\beta}$\\
	}
	
	\bigskip
	{\bf{STEP 4}} {}
	\\
	\eIf{$\mathscr{R} \backslash UC = \varnothing$}{
		NF system has a weakly reversible CF-decomposition 		
	}{
		$RY = \varnothing$\\
		$RYS = (\mathscr{R} \backslash UC) \backslash RY$\\
		\For{$r \in RYS$}
		{
			\eIf{there is a cycle $r(y)$ of $y$ containing $r$ such that its branching reaction belongs to CF$_j$ and no reactions from
				$\bigcup_{i \ne j} {\left( {{\mathscr{R}_{0,i}}} \right)}$ are contained in $r(y)$}{
				add $r(y)$ to subnetwork with reaction set $\mathscr{R}_{0,\beta}$ and with the same CF$_\beta$\\
				$RY=$ set of all reactions in $r(y)$\\
				
			}{the NF system has no weakly reversible CF-decomposition}
			
		}
	}
	\bigskip
	{\bf{OUTPUT}} 
	\\
	{weakly reversible CF-decomposition if it exists}\\
	\bigskip	
	\caption{An algorithm to find a weakly reversible CF-decomposition of an NF system with a single linkage class and a single NF node}
	\label{algorithm0}
\end{algorithm}

\begin{example}
	In the Schmitz's model subnetwork in Example \ref{ex:schmitz}, it is easy to check that the CF-cycle subnetwork covering is the CF-decomposition.
	\label{ex:schmitz:WR:dec}
\end{example}

\begin{example}
	Consider the CRN in Figure \ref{fig:Example2WRdec}.
	\begin{figure}
		\begin{center}
			\includegraphics[width=22cm,height=7cm,keepaspectratio]{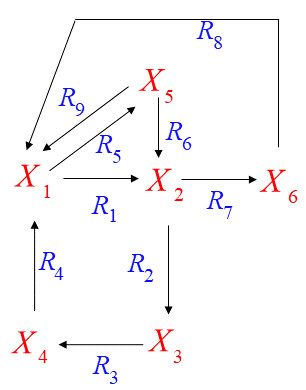}
			\caption{Reaction Network in Example \ref{ex:Example2WRdec}}
			\label{fig:Example2WRdec}
		\end{center}
	\end{figure}
	The complex $X_1$ is an NF node with CF-subsets CF$_1(X_1)=\{R_1\}$ and $\{R_5\}$. This linkage class is not Eulerian since 
	$\deg_-(X_1) = 3$ and $\deg_+(X_1) = 2$. On the other hand,
	$CFC_1(X_1)$ contains the cycles: $$\{R_1,R_2,R_3,R_4\} {\text{ and }} \{R_1,R_7, R_8\}.$$
	In addition, $CFC_2(X_1)$ contains the cycles $$\{R_5,R_6,R_7,R_8\}, \{R_5,R_6, R_2,R_3,R_4\} {\text{ and }} \{R_5,R_9\}.$$
	We have $\mathscr{CFC}_1(X_1)=\mathscr{N}_{0,1}$ and $\mathscr{CFC}_2(X_1)=\mathscr{N}_{0,2}$ as given in Figure \ref{fig:Example2cfc}.
	\begin{figure}
		\begin{center}
			\includegraphics[width=27cm,height=9cm,keepaspectratio]{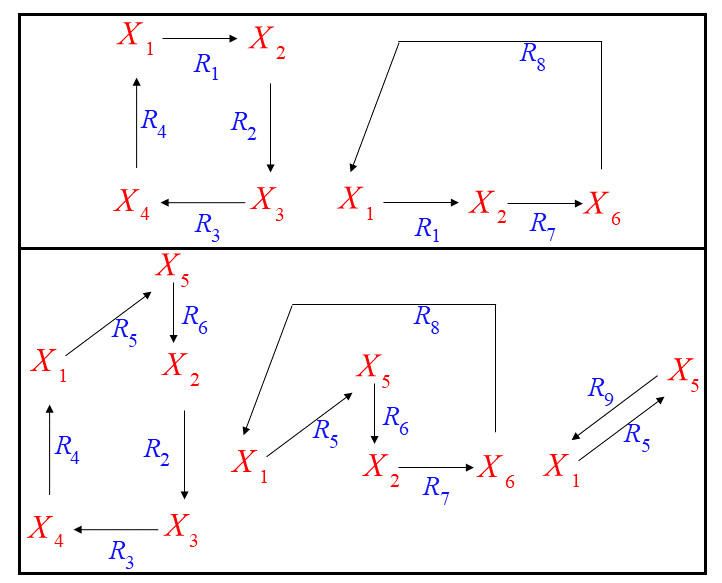}
			\caption{Subnetworks $\mathscr{CFC}_1(X_1)=\mathscr{N}_{0,1}$ on the upper portion and $\mathscr{CFC}_2(X_1)=\mathscr{N}_{0,2}$ on the lower portion in Example \ref{ex:Example2WRdec}}
			\label{fig:Example2cfc}
		\end{center}
	\end{figure}
	The weakly reversible CF-decomposition of the network is given by $$\{R_1,R_2,R_3,R_4\} {\text{ and }} \{R_5,R_6,R_7,R_8, R_9\}$$ consisting of a cycle and the union of two cycles.
	\label{ex:Example2WRdec}
\end{example}

\section{An algorithm for determining the existence of a weakly reversible CF-decomposition of an NF system}
\label{sec:algo:wr}

We generalize Algorithm \ref{algorithm0} of obtaining a weakly reversible CF-decomposition of an NF system for a finite number of NF nodes with the following steps and in Algorithm \ref{algorithm1}.

\begin{enumerate}
	\item We input the reaction network and the kinetic order matrix.
	\item We identify the linkage classes from the given network.
	\item For each linkage class $\mathscr{L}_\alpha$, where $\alpha = 1,2, \ldots, \ell$, check if there is an NF-node.
	If there is no NF-node then the linkage class is the trivial decomposition that is weakly reversible. If there is at least one NF-node, then let $RC$ be the set of all NF-nodes.
	\item For $\theta=1,2,\ldots,|RC|$, treat the $\theta$th NF-node as the only NF-node $y$ in this step. We then use Algorithm \ref{algorithm0} of finding a weakly reversible CF-decomposition with a single NF node. We repeat STEP 3 until all the linkage classes are exhausted.
\end{enumerate}

\begin{algorithm}
	\SetAlgoLined
	\bigskip
	{\bf{STEP 1}} {Inputs}\\
	{\bf{Input 1}} {reaction network $\mathscr{N}$}\\
	{\bf{Input 2}} {kinetic order matrix $F$}\\
	\bigskip
	
	{\bf{STEP 2}} {Identify the linkage classes
		$\mathscr{L}_\alpha$ for $\alpha = 1,2, ... , \ell$.}\\
	
	\bigskip
	{\bf{STEP 3}} {Compute for the set $BN$ of branching nodes and the set $RC$ of NF nodes}\\
	\For{$\alpha = 1$ to $\ell$}{
		$r(x) =$ set of all reactions in $\mathscr{L}_\alpha$ with $x$ as reactant complex\\
		$BN=\{x : |r(x)|  > 1\}$\\
		$RC = \{ x \in BN:{\text{at least one kinetic order value associated to}}$\\
		${\text{the branching reactions are different}}\} 
		$\\
		\eIf{$|RC|=0$}{
			$\mathscr{L}_\alpha$ is included for possible weakly reversible decomposition
		}{
			proceed to STEP 4
		}
	}
	\bigskip
	{\bf{STEP 4}} {Analysis for at least one NF node}\\
	\For{$\theta = 1$ to $|RC|$}{
		just in this step, treat the $\theta$th NF node as the only NF node in this particular iteration\\
		for $\theta \ge 2$, proceed if the network was not decomposed into CF-subsystems with respect to $\theta$th NF-node resulted from the previous or $\theta -1$ iteration\\
		proceed to STEPS 2, 3, and 4 of Algorithm \ref{algorithm0}\\
		update the decomposition for this iteration which is treated with single NF-node
	}
	
	\bigskip
	{\bf{OUTPUT}} 
	\\
	{weakly reversible CF-decomposition of the NF system, if it exists}\\
	\bigskip
	\caption{How to find a weakly reversible CF-decomposition of an NF system, if it exists}
	\label{algorithm1}
\end{algorithm}

\begin{example}
	Consider the reaction network in Figure \ref{pic:algoex1}. The kinetic orders of the branching reactions are indicated on the arrows. We can easily identify which of the branching nodes are NF nodes and CF nodes.
	
	\begin{figure}
		\begin{center}
			\includegraphics[width=22cm,height=7cm,keepaspectratio]{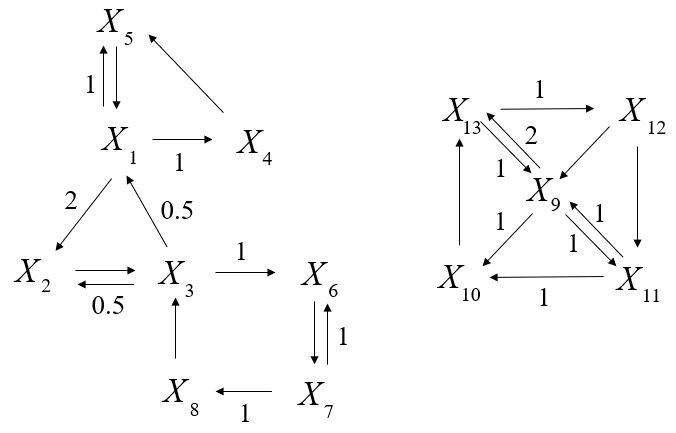}
			\caption{Network in Example \ref{ex:algoex1}}
			\label{pic:algoex1}
		\end{center}
	\end{figure}
	
	The algorithm is run for each linkage class. Firstly, consider the linkage class $\mathscr{L}_1$ on the left. The NF nodes are $X_1$ and $X_3$.
	For $X_1$, there are 2 CF-subsets. The first one contains $X_1 \to X_4$ and $X_1 \to X_5$. The second contains $X_1 \to X_2$. Hence, we have the weakly reversible decomposition in Figure \ref{pic:algoex1:dec1} treating $X_3$ as a CF-node in this iteration:

	\begin{figure}
		\begin{center}
			\includegraphics[width=12cm,height=4cm,keepaspectratio]{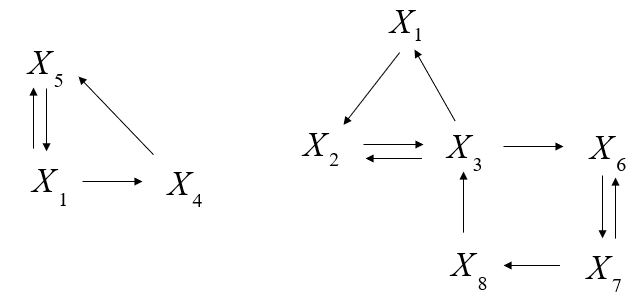}
			\caption{Decomposition of the left linkage class with respect to node $X_1$ in Example \ref{ex:algoex1}}
			\label{pic:algoex1:dec1}
		\end{center}
	\end{figure}
	
	We proceed with the next NF node $X_3$. There are 2 CF-subsets. The first has the reaction $X_3 \to X_6$ while the other CF-subset has the reaction $X_3 \to X_2$. We have already exhausted all NF-nodes. We have a weakly reversible RDK decomposition of $\mathscr{L}_1$ in Figure \ref{pic:algoex1:dec2}.
	
	\begin{figure}
		\begin{center}
			\includegraphics[width=12cm,height=4cm,keepaspectratio]{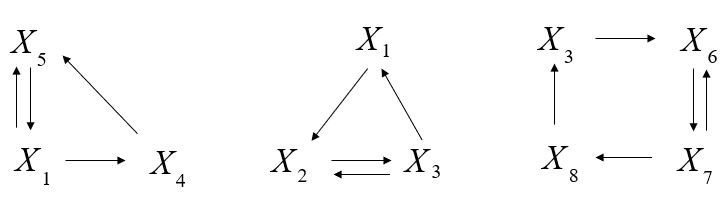}
			\caption{Decomposition of the left linkage class with respect to node $X_1$ and then node $X_3$ in Example \ref{ex:algoex1}}
			\label{pic:algoex1:dec2}
		\end{center}
	\end{figure}
	
	We now consider the second linkage class $\mathscr{L}_2$. The only NF node is $X_9$. We obtain a weakly reversible decomposition of $\mathscr{L}_2$ in Figure \ref{pic:algoex1:dec3}.
	
	\begin{figure}
		\begin{center}
			\includegraphics[width=10cm,height=3.5cm,keepaspectratio]{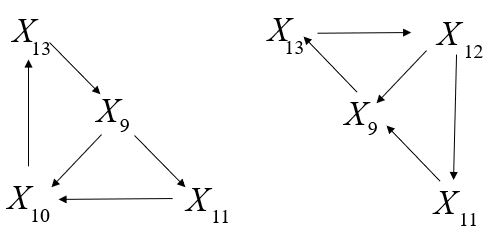}
			\caption{Decomposition of the right linkage class with respect to node $X_9$ in Example \ref{ex:algoex1}}
			\label{pic:algoex1:dec3}
		\end{center}
	\end{figure}
	
	Thus, we have a weakly reversible decomposition of the whole network in Figure \ref{pic:algoex1:decfinal}.

	\begin{figure}
		\begin{center}
			\includegraphics[width=21cm,height=7cm,keepaspectratio]{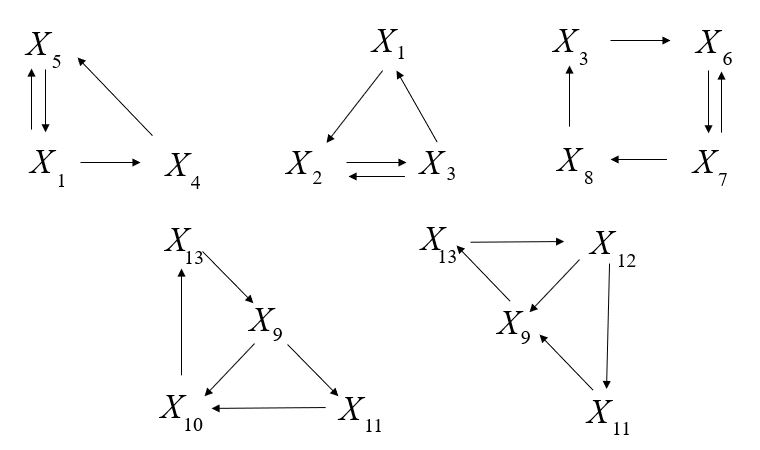}
			\caption{A weakly reversible CF-decomposition in Example \ref{ex:algoex1}}
			\label{pic:algoex1:decfinal}
		\end{center}
	\end{figure}
	
	\label{ex:algoex1}
	
\end{example}

\begin{example}
	Consider the CRN in Figure \ref{pic:algoex2} with indicated kinetic order values for the branching reactions. There is only one linkage class.
	For the NF node $X_1$, there are two CF subsets, say CF$_1(X_1)$ that has a cycle of two reactions and CF$_2(X_1)$ that has a cycle of five reactions.
	Since it is not possible to choose a cycle in CF$_1$ which has no common reaction with CF$_2$, a CF-decomposition is not possible.
	Indeed, we can check in STEP 3 of Algorithm \ref{algorithm0}, which is inside STEP 4 of Algorithm \ref{algorithm1} that 
	the system has no weakly reversible CF-decomposition.  
	
	\begin{figure}
		\begin{center}
			\includegraphics[width=9cm,height=4cm,keepaspectratio]{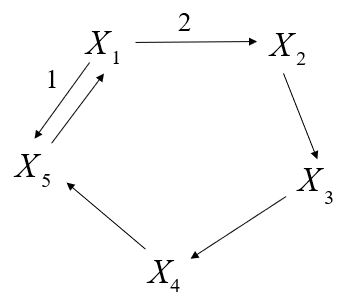}
			\caption{Network in Example \ref{ex:algoex2}}
			\label{pic:algoex2}
		\end{center}
	\end{figure}	
	\label{ex:algoex2}	
\end{example}

\section{Method for constructing independent weakly reversible CF-decompositions of an NF system}
\label{sec:construct:independent}
In this section, we present how we can determine weakly reversible CF-decompositions of an NF system that are independent. Before we proceed with our results, we recall the following discussion in the work of Hernandez and De la Cruz \cite{hernandez:delacruz1}, which gives a way of finding independent decompositions. One reason why we want to find independent decompositions is due to the Feinberg Decomposition Theorem (Theorem \ref{feinberg:decom:thm}), as this result provides equality between the set of equilibria of the whole network and the intersection of the sets of equilibria of the whole network, for network with independent decompositions.

\begin{definition}
	\label{def:coordinate:graph}
	Let ${\bf R}=\{{\bf R}_1,\ldots, {\bf R}_m\}$ be a set of vectors such that the span of ${\bf R}$ is of dimension $p$. Suppose that $\{{\bf R}_1,\ldots,{\bf R}_p\} \subseteq {\bf R}$ is linearly independent.
	The {\bf coordinate graph} of ${\bf R}$ is the graph $G=(V,E)$ with vertex set $V=\{v_1,\ldots, v_p\}$ and edge set $E$ such that
	$(v_i,v_j)$ is an edge in $E$ if and only if there exists $k >p$ with ${\bf R}_k=\displaystyle \sum_{j=1}^p a_j {\bf R}_j$, and both $a_i$ and $a_j$ are nonzero.
\end{definition}

Recall that a graph $G$ is connected if each pair of vertices is joined by a path in $G$. In addition, a connected component of $G$ is a maximal connected subgraph of $G$, and that $G$ is connected if and only if it has exactly
one connected component. In addition, there is no path that connects two vertices in different components.

The following is the main theorem in \cite{hernandez:delacruz1}, which is the basis of the method in the paper to find independent decompositions of reaction networks.

\begin{theorem} Let  ${\bf R}$ be a finite set of vectors. A nontrivial independent decomposition of ${\bf R}$ exists if and only if 
	the coordinate graph of ${\bf R}$ is not connected. 
\end{theorem}

\begin{remark}
	In the proof of this theorem, the reactions of the network correspond to their reaction vectors, and the reaction network is being converted into coordinate graph. The idea is decomposing the original reaction network into subnetworks, which is induced by the partition that happens in the coordinate graph. This result can be used to get the ``finest'' independent decomposition (i.e., you can no longer decompose it further into independent decomposition with a greater number of subnetworks). It can be deduced that the only independent decompositions are the finest decomposition and the coarsenings of this finest independent decomposition \cite{amistas}.
\end{remark}

We now have the following result:

\begin{proposition}
	An NF system has an independent weakly reversible CF-decomposition if and only if the coarsest weakly reversible CF-decomposition is independent.
\end{proposition}

\begin{proof}
	The right to left direction is obvious. For the other direction, suppose the NF system has an independent weakly reversible CF-decomposition. The conclusion follows from the fact that the coarsening of an independent decomposition is also independent.
\end{proof}

The coarsest is easier to identify using the methods of finding weakly reversible and independent decomposition. We can use the definition to check if the coarsest weakly reversible CF-decomposition is independent.

\begin{example}
	Consider the Schmitz's subnetwork in Example \ref{ex:schmitz}. The coarsest weakly reversible CF-decomposition induced by the partition 
	\[\left\{ {\left\{ {{R_1},{R_2},{R_3},{R_4}} \right\},\left\{ {{R_5},{R_6},{R_7},{R_8}} \right\}} \right\}\]
	was given in Example \ref{ex:schmitz:WR:dec}. The sum of the ranks of the two subnetworks is the same as the rank of the network, i.e., $2+3=5$. By definition, the decomposition is independent. Therefore, the decomposition is an independent weakly reversible CF-decomposition.
\end{example}

\section{Positive equilibria of power law systems with a weakly reversible PL-RDK decomposition of PLP type}
\label{sec:existence:parametrization}
In this Section, we demonstrate how the availability of weakly reversible CF-decompositions leads to new results on equilibria for NF systems in the case of  power law systems. We derive a broad generalization of the Deficiency Zero Theorem of Fortun et al. \cite{fortun2} to systems with positive deficiency, thus showing that the essential property for the result in \cite{fortun2} is the availability of a particular decomposition and not the deficiency value. The generalization was originally formulated only for PL-NDK systems, but we then observed that the proof did not use the NDK property and that, when asserted more generally for power law systems, special cases coincided with those of a theorem of B. Boros for mass action systems \cite{BORO2013} and a theorem of Talabis et al. about PL-TIK ($\widehat{T}$-rank maximal kinetics) systems \cite{TAM2018}. For details about PL-TIK, we refer the reader to pages 371-372 of \cite{TAM2018}. We also formulate the corresponding Theorem for the set of complex balanced equilibria of power law systems.

\subsection{The reaction network of kinetic complexes $\widetilde{\mathscr{N}}$ of a PLK system $(\mathscr{N}, K)$}

S. M\"uller and G. Regensburger introduced the concepts of ``kinetic order subspace'' and ``kinetic deficiency'' for cycle terminal PL-RDK systems in the context of their theory of generalized mass action systems (GMAS) in two papers in 2012 and 2014 \cite{MURE2014}. In this Section, we present an extension of the concepts to cycle terminal PLK systems, which for the subset of PL-RDK systems coincide with the M\"uller-Regensburger definition. We use the term ``kinetic flux subspace'' instead of their ``kinetic order subspace''.

We denote with $\mathscr{R}(y)$ the set of (branching) reactions of $y$ and $| \mathscr{R}(y) | = \deg_+(y)$. Also, $\mathscr{R}_i(y)$ stands for a CF-subset of $y$, $i = 1,..., N_R(y)$.

\begin{definition}
	If $y$ is a complex of a cycle terminal PLK system $(\mathscr{N}, K)$, then $\mathscr{\widetilde{C}}(y) := \{ F_q | q \in \mathscr{R}(y)\}$
	is the set of kinetic complexes of $y$, and $\widetilde{y}$ denotes an element of the set. For any reaction $q: y \to y'$, $$\mathscr{\widetilde{R}}(q) := \{ \widetilde{y} \to \widetilde{y'}|
	\widetilde{y} \in \mathscr{\widetilde{C}}(y), \widetilde{y'} \in \mathscr{\widetilde{C}}(y')\}$$ is the set of kinetic complex reactions of $q$, and $\widetilde{q}$ denotes an element of the set.
\end{definition}

The following Proposition shows some basic properties of these sets:

\begin{proposition}
	Let (N, K) be a cycle terminal PLK system. Then
	\begin{enumerate}
		\item[i.] For each complex $y$, $| \mathscr{\widetilde{C}}(y) | = N_R(y)$.
		\item[ii.] For each reaction $q: y \to y'$, $| \mathscr{\widetilde{R}}(q) |= N_R(y) N_R(y') - | \mathscr{\widetilde{C}} (y) \cap \mathscr{\widetilde{C}} (y')|$.
	\end{enumerate}
	\label{cycleterminal:formulas}
\end{proposition}

\begin{proof}
	This follows directly from the definitions.
\end{proof}

We now define the reaction network $\mathscr{\widetilde{N}}$ of kinetic complexes of a cycle terminal PLK system $(\mathscr{\widetilde{N}}, K)$:
\begin{definition}
	The set of kinetic complexes has a reaction network structure given by $\mathscr{\widetilde{N}}= (\mathscr{S}, \mathscr{\widetilde{C}}, \mathscr{\widetilde{R}})$
	with $\mathscr{\widetilde{C}}= \bigcup\limits_y   \mathscr{\widetilde{C}} (y)$ 
	and $\mathscr{\widetilde{R}}= \bigcup\limits_q   \mathscr{\widetilde{R}} (q)$ .
	We denote $| \mathscr{\widetilde{C}} |$ and $| \mathscr{\widetilde{R}} |$ with $\widetilde{n}$ and $\widetilde{r}$, respectively.
\end{definition}
\begin{remark}
	If for a reaction $y \to y'$, $\mathscr{\widetilde{C}} (y) \cap \mathscr{\widetilde{C}} (y') \ne \varnothing$, loops may result in general, and our convention is that we remove them. This will not happen if the kinetics is interaction span surjective. We shall denote the set of interaction span surjective kinetics with PL-ISK. For PL-RDK systems, PL-ISK = PL-FSK (with factor span surjective kinetics).
\end{remark}

The following Proposition collects the basic properties of $\mathscr{N}$:

\begin{proposition}
	Let $(\mathscr{N}, K)$ be a cycle terminal PLK system. Then
	\begin{enumerate}
		\item[i.] $\mathscr{\widetilde{N}}$ is cycle terminal. If $\mathscr{N}$ is weakly reversible, then $\mathscr{\widetilde{N}}$ is also weakly reversible.
		\item[ii.] $\widetilde{n} \le N_R$ and $\widetilde{r} \le 
		\sum\limits_{y \to y'} N_R(y) N_R(y')$. If $K \in$ PL-ISK, then equality holds in both relationships, and $l = \widetilde{l}$.
	\end{enumerate}
\end{proposition}
\begin{proof}
	For (i.), as the notation already indicates, any kinetic complex $\widetilde{y}$ derives from a reactant, ensuring the existence of a reaction in $\mathscr{\widetilde{R}}$ of which it is the reactant. In the weakly reversible case, any reaction in $\mathscr{R}$ is in a cycle, leading to every reaction in $\mathscr{R}$ being in a cycle too. For (ii.), the upper bounds are derived from taking the sum over all reactants and over all reactions. If $y$ and $y'$ are distinct complexes, since the kinetics is PL-ISK, their reactions are disjoint, and hence, $\mathscr{\widetilde{C}} (y) \cap \mathscr{\widetilde{C}} (y') =  \varnothing$, and the formula follows from Proposition \ref{cycleterminal:formulas}.
\end{proof}

We denote the incidence, complexes and stoichiometric maps of $\mathscr{\widetilde{N}}$ as $\widetilde{I}_a$ , $\widetilde{Y}$ and $\widetilde{N}$, respectively.

\begin{definition}
	The {\bf kinetic flux space} $\widetilde{S}$ of $(\mathscr{N}, K)$ is defined as the image of $\widetilde{N}$ and its dimension is denoted as the kinetic rank $\widetilde{s}$.
\end{definition}

Thus, the kinetic flux space $\widetilde{S}$ is just the stoichiometric subspace of the network of kinetic complexes.

\begin{definition}
	The {\bf kinetic complex deficiency} of $(\mathscr{N}, K)$ is defined as $$\delta_\mathscr{\widetilde{N}}:= \widetilde{n} - \widetilde{l} - \widetilde{s}.$$
\end{definition}

In view of the geometric interpretation, this value is nonnegative. We have the following important fact:

\begin{proposition}
	If $\widetilde{n} - \widetilde{l} = n - l$ for a cycle terminal PLK system $(\mathscr{N}, K)$, then $\delta_\mathscr{\widetilde{N}} = \widetilde{\delta}$.  This holds in particular for PL-FSK systems. 
\end{proposition}

\begin{proof}
	Due to the graph isomorphism for PL-FSK systems, we have $\widetilde{n} = n$ and $\widetilde{l} = l$.
\end{proof}

\begin{remark}
	For PL-RDK systems which are not factor span surjective, we only have $\widetilde{n} \le n$ and $\widetilde{l} \le l$ so that no general statement about their differences is possible. It is in view of this restricted coincidence with the kinetic deficiency that we introduced a different name and symbol for kinetic complex deficiency.
\end{remark}

\begin{example}
	We consider again the subnetwork of Schmitz's carbon cycle model from our Running Example \ref{ex:schmitz}.
	The system has 7 kinetic complexes, 1 linkage class and kinetic rank = 6. Hence, its kinetic complex deficiency = 0. This is also the kinetic deficiency of the weakly reversible PL-RDK system, so that from the M\"uller-Regensburger theory, $(\mathscr{\widetilde{N}}, \widetilde{K})$ is unconditionally complex balanced.
\end{example}

\subsection{The subnetwork $\widetilde{\mathscr{N}}_\mathscr{D}$ of a PL-RDK decomposition of a PL-NDK system $(\mathscr{N}, K)$}

In this Section, we use a framework to show that the flux subspaces belong to a subnetwork of kinetic complexes (determined by the decomposition) whose deficiency equals that of the original network. We introduce the construction in a more general context, and then focus on a PLK system $(\mathscr{N}, K)$ with a weakly reversible, bi-independent PL-FSK decomposition $\mathscr{D}: \mathscr{N}   = \mathscr{N}_1 \cup \ldots \cup \mathscr{N}_k$ with $\dim S_i = \dim \widetilde{S}_i$.

We set $\widetilde{\mathscr{N}}_i$ := subnetwork
of 
$\widetilde{\mathscr{N}}$
induced by ${\mathscr{N}}_i$, i.e., take the reactions defining it, then form the kinetic complexes (still in the sense of M\"uller-Regensburger).

\begin{definition}
	The {\bf induced subnetwork} $\widetilde{\mathscr{N}}_{\mathscr{D}}$  is defined as the union of the $\widetilde{\mathscr{N}}_i$. If the covering $\{\widetilde{\mathscr{N}}_i\}$ is a decomposition, we call it the {\bf induced decomposition}. 
\end{definition}

In the following, we will assume that the covering is indeed a decomposition (so that in the examples, this has to be verified). Note however, that the independence property, i.e., $S$ is the direct sum of the $S_i$'s actually implies that the covering is a decomposition.

We now introduce a convenient notation for our results: we say that a decomposition property is {\bf bi-level} if the property holds for both the network decomposition and the induced subnetwork decomposition. We will deal mostly with bi-level weakly reversible, bi-level independent, and bi-level bi-independent decompositions.

\begin{proposition}
	Let $(\mathscr{N}, K)$ be a PLK system with a weakly reversible PL-RDK decomposition $\mathscr{D}: \mathscr{N} = \mathscr{N}_1 \cup \ldots \cup \mathscr{N}_k$ with $\dim S_i = \dim \widetilde{S}_I$.
	\begin{enumerate}
		\item[i.] $\mathscr{N}$ is weakly reversible.
		\item[ii.] If $\mathscr{D}$ is an independent decomposition, then it is bi-level independent.
		\item[iii.]	If $\mathscr{D}$ is a bi-independent PL-FSK decomposition, then it is bi-level weakly reversible and bi-level bi-independent. If 
		$\widetilde{n}_\mathscr{D}$ and $\widetilde{l}_\mathscr{D}$ 
		denote the number of complexes and linkage classes of $\widetilde{\mathscr{N}}_\mathscr{D}$ respectively, then  $\widetilde{n}_\mathscr{D} - \widetilde{l}_\mathscr{D} = n - l$,
		$\widetilde{s}_\mathscr{D} = s$, and hence, $\delta(\widetilde{\mathscr{N}}_\mathscr{D}) = \delta (\mathscr{N})$.
	\end{enumerate}
\end{proposition}

\begin{proof}
	(i.) Any reaction of $\mathscr{D}$ is in one of the subnetworks, and hence, in a cycle. (ii.) $\dim S_i = \dim \widetilde{S}_I$ is equivalent to $S_i$ and $\widetilde{S}_I$ being isomorphic, hence $S_i \cap S_j = \{0\} \implies \widetilde{S}_i \cap \widetilde{S}_j = \{0\}$.
	(iii.) Note that each reaction in the subnetwork $\mathscr{N}_i$ is in a cycle. Hence, any reaction in the subnetwork $\widetilde{\mathscr{N}}_i$ is also in a cycle. Thus, the $\widetilde{\mathscr{N}}_i$ also form a weakly reversible decomposition.
	Independence of the induced decomposition follows from (ii.) and the fact that $\mathscr{D}$ is independent.
	On the other hand, $\sum (n_i-l_i)=n-l$ follows from the assumption of incidence independence of $\mathscr{D}$. The graph isomorphism of the corresponding subnetworks of $\mathscr{N}$ and $\widetilde{\mathscr{N}}_\mathscr{D}$ implies $\widetilde{n}-\widetilde{l}=n-l$, and $\widetilde{s}_\mathscr{D} = s$ follows from (ii.). Hence, $$\delta(\widetilde{\mathscr{N}}_\mathscr{D}) = \widetilde{n}_\mathscr{D} - \widetilde{l}_\mathscr{D} - \widetilde{s}_\mathscr{D}= n - l -s = \delta (\mathscr{N}). $$
\end{proof}

\begin{example}
	For the subnetwork of Schmitz's model, we have $\widetilde{n}_\mathscr{D} = 7$, $\widetilde{l}_\mathscr{D} = 2$, $n= 6$ , $l = 1$, and $7 - 2 = 6 - 1$.  Furthermore, $\widetilde{s}_\mathscr{D} = s = 5$. We emphasize that the graph isomorphism is not between the networks $\mathscr{N}$ and $\widetilde{\mathscr{N}}_\mathscr{D}$, but among the corresponding subnetworks via the decomposition described in the proposition given above. We notice that in this example, $\widetilde{n}_\mathscr{D} \ne n$ and $\widetilde{l}_\mathscr{D} \ne l$ but $\widetilde{n}_\mathscr{D} - \widetilde{l}_\mathscr{D} = n-l$.
\end{example}

\subsection{An extension of the Deficiency Zero Theorem of Fortun et al. to PLK systems with positive deficiency}

To formulate the extension of the result of Fortun et al. \cite{fortun2}, we need one more concept which generalizes the parametrization property of PL-RLK (power law with reactant set linear independent kinetics) systems with low deficiency. For details about PL-RLK, we refer the reader to pages 371-372 of \cite{TAM2018}, and pages 627-628 of \cite{fortun2}.

\begin{definition}
	A kinetic system $(\mathscr{N},K)$ is of type {\bf PLP (positive equilibria log-parametrized)} if 
	\begin{enumerate}
		\item[i.] $E_+(\mathscr{N},K) \ne \varnothing$, and 
		\item[ii.] $E_+(\mathscr{N},K) = \{ x \in \mathbb{R}^\mathscr{S}_{>0}| \log x - \log x^* \in (P_E) ^\perp\}$,
	\end{enumerate}
	where $P_E$ is a subspace of $\mathbb{R}^\mathscr{S}$ and $x^*$ is a positive equilibrium.
\end{definition}

\begin{definition}
	A kinetic system $(\mathscr{N}, K)$ is of type {\bf CLP (complex balanced equilibria log parametrized)} if
	\begin{enumerate}
		\item[i.] $Z_+(\mathscr{N},K) \ne \varnothing$, and 
		\item[ii.] $Z_+(\mathscr{N},K) = \{ x \in \mathbb{R}^\mathscr{S}_{>0}| \log x - \log x^* \in (P_Z) ^\perp\}$,
	\end{enumerate}
	where $P_Z$ is a subspace of $\mathbb{R}^\mathscr{S}$ and $x^*$ is a complex balanced equilibrium.
\end{definition}

A kinetic system is bi-LP if it is of PLP and of CLP type and $P_E = P_Z$.  We will often use the shorter PLP system, CLP system and bi-LP system notation as well as the collective term ``LP systems''.

A key property of an LP system was in principle already derived by M. Feinberg in his 1979 lectures \cite{feinberg} as shown in \cite{MAME2021}:

\begin{theorem}
	Let $(\mathscr{N},K)$ be a chemical kinetic system.
	\begin{enumerate}
		\item[i.] If $(\mathscr{N},K)$ is a PLP system, then $|E_+(\mathscr{N}, K) \cap Q| = 1$ for any positive coset $Q$ of $P_E$ in $\mathbb{R}^\mathscr{S}$.
		\item[ii.] If $(\mathscr{N},K)$ is a CLP system, then $|Z+(\mathscr{N}, K) \cap Q| = 1$ for any positive coset $Q$ of $P_Z$ in $\mathbb{R}^\mathscr{S}$.
		\item[iii.] If $(\mathscr{N},K)$ is a bi-LP system, then it is absolutely complex balanced, i.e., each positive equilibrium is complex balanced.
	\end{enumerate}
\end{theorem}

The generalization of the main result of Fortun et al. is the following Theorem:

\begin{theorem}
	Let $(\mathscr{N},K)$ be a power law system with a weakly reversible PL-RDK decomposition $\mathscr{D}: \mathscr{N}   = \mathscr{N}_1 \cup \mathscr{N}_2  \cup \ldots \mathscr{N}_k$. If $\mathscr{D}$ is bi-level independent and of PLP type with $P_{E,i} = \widetilde{S}_i$, then $(\mathscr{N},K)$ is a weakly reversible PLP system with $P_E = \sum \widetilde{S}_i$.
\end{theorem}

\begin{proof}
	Each subnetwork of $(\mathscr{N},K)$ in the decomposition is weakly reversible and hence, the network is weakly reversible.
	For the last two statements, we prove the case when $k=2$, and the general case can be proven inductively.
	Since each subnetwork of $(\mathscr{N},K)$ has the PLP-type, we have $E_+(\mathscr{N}_1,K_1) \ne \varnothing$ and $E_+(\mathscr{N}_2,K_2) \ne \varnothing$. In addition, 
	$$E_+(\mathscr{N}_1,K_1) = \{ x \in \mathbb{R}^\mathscr{S}_{>0}| \log x - \log x_1 \in \widetilde{S}_1 ^\perp\}$$ and 
	$$E_+(\mathscr{N}_2,K_2) = \{ x \in \mathbb{R}^\mathscr{S}_{>0}| \log x - \log x_2 \in \widetilde{S}_2 ^\perp\}$$
	where $x_1 \in E_+(\mathscr{N}_1,K_1)$ and $x_2 \in E_+(\mathscr{N}_2,K_2)$.
	The independence of the decomposition and Theorem \ref{feinberg:decom:thm} (Feinberg Decomposition Theorem) yield
	$$E_+(\mathscr{N},K)=E_+(\mathscr{N}_1,K_1) \cap E_+(\mathscr{N}_2,K_2).$$
	Thus, $x^* \in E_+(\mathscr{N},K) \iff x^* \in E_+(\mathscr{N}_1,K_1) \cap E_+(\mathscr{N}_2,K_2)$, and equivalently, $$\log x^* \in \left[\log x_1 + \widetilde{S}_1 ^\perp\right] \cap \left[\log x_2 + \widetilde{S}_2 ^\perp\right].$$
	From properties of cosets,
	$$\left[\log x_1 + \widetilde{S}_2 ^\perp\right] \cap \left[\log x_2 + \widetilde{S}_2 ^\perp\right] \ne \varnothing \iff \log x_1 - \log x_2 \in \widetilde{S}_2 ^\perp + \widetilde{S}_2 ^\perp.$$
	Independence of the induced decomposition ensures $\widetilde{S}_1 \cap \widetilde{S}_2 = \{0\}.$\\
	Then $\widetilde{S}_1^\perp+\widetilde{S}_2^\perp=\left(\widetilde{S}_1 \cap \widetilde{S}_2\right)^\perp=\{0\}^\perp=\mathbb{R}^\mathscr{S}$. Thus,
	$$\left[\log x_1 + \widetilde{S}_1 ^\perp\right] \cap \left[\log x_2 + \widetilde{S}_2 ^\perp\right] \ne \varnothing.$$ Let $\widehat x \in \left[\log x_1 + \widetilde{S}_1 ^\perp\right] \cap \left[\log x_2 + \widetilde{S}_2 ^\perp\right]$ and take $x^*=e^{\widehat{x}}$.\\We have ${\widehat{x}} \in \left[\log x_1 + \widetilde{S_1} ^\perp\right]$ and ${\widehat{x}} \in \left[\log x_2 + \widetilde{S_2} ^\perp\right]$ $\implies$ $x^* \in E_+(\mathscr{N}_1,K_1)$ and $x^* \in E_+(\mathscr{N}_2,K_2)$ $\implies$ $x^* \in E_+(\mathscr{N},K)$.
	Then,
	\[\left[\log x_1 + \widetilde{S}_1 ^\perp\right] \cap \left[\log x_2 + \widetilde{S}_2 ^\perp\right] = \log x^* + \left[ \widetilde{S}_1^\perp \cap \widetilde{S}_2^\perp \right]=\log x^* +\left(\widetilde{S_1} + \widetilde{S_2} \right)^\perp.\]
	Hence,
	$E_+(\mathscr{N},K) = \left\{ x \in \mathbb{R}^\mathscr{S}_{>0}| \log x - \log x^* \in \left(\widetilde{S}_1+\widetilde{S}_2\right) ^\perp\right\}.$
\end{proof}

\begin{corollary}
	Let $(\mathscr{N}, K)$ be a PLK system with a weakly reversible PL-RDK decomposition $\mathscr{D}: \mathscr{N}   = \mathscr{N}_1 \cup \mathscr{N}_2  \cup \ldots \mathscr{N}_k$ with $\dim S_i = \dim \widetilde{S}_I$. If $\mathscr{D}$ is independent of PLP type, then $(\mathscr{N}, K)$ is a weakly reversible PLP system with $P_E = \sum \widetilde{S}_i$.
\end{corollary}
\begin{proof}
	We already showed in the previous section that these properties imply bi-level independence.
\end{proof}

\begin{example}
	We consider the running example in Example \ref{ex:schmitz}, i.e., the subnetwork of Schmitz's carbon cycle model is a PL-NDK to illustrate the Corollary. The weakly reversible independent decomposition has $\dim S_1 = \dim \widetilde{S}_1 = 2$ and
	$\dim S_2 = \dim \widetilde{S}_2 = 3$. According to the Deficiency Zero Theorem for PL-TIK systems \cite{TAM2018}, PL-RLK systems are of PLP type.
\end{example}

\begin{example}
	Any weakly reversible PL-RLK system satisfying the conditions of the Deficiency One Theorem for PL-TIK systems \cite{TAM2018} with at least one linkage class with deficiency = 1, illustrates the Theorem, but not the Corollary. For the linkage class with deficiency = 1, 
	$s_i = n_i - 1 - 1 = n_i-2$, while $\widetilde{s}_i = n_i - 1 - 0=n_i-1$. Thus, although the linkage class decomposition is bi-level bi-independent and PL-FSK, $\delta(\mathscr{N}_\mathscr{D}) = 0$
	while $1 \le \delta(\mathscr{N}) \le l$.
\end{example}

\begin{example}	
	In his PhD thesis \cite{BORO2013}, B. Boros states the following Theorem: \end{example}

\begin{theorem}
	For a mass action system $(\mathscr{}N, K)$ with independent linkage classes, $E_+(\mathscr{L}_i, K_i) \ne \varnothing \iff E_+(\mathscr{N}, K) \ne \varnothing$. In the complex balanced case, i.e. all linkage classes are complex balanced and hence weakly reversible, we have $E_+(\mathscr{L}_i, K_i) = Z_+(\mathscr{L}_i, K_i)$, and hence, of PLP type with $P_E = S$, according to classical results of Horn and Jackson. Since mass action systems are all factor span surjective and $S_i = \widetilde{S}_I$, we see that this is a special case of the Corollary for a bi-level bi-independent decomposition.
\end{theorem}

\begin{example}
	Talabis et al. \cite{TAM2018} derived the following Theorem 4: 	
\end{example}

\begin{theorem}
	Let $(\mathscr{S},\mathscr{C},\mathscr{R},K)$ be PL-TIK.If for each linkage class subnetwork $\mathscr{L}_i$, $E_+(\mathscr{L}_i, K) \ne \varnothing$ then $E_+(\mathscr{N}, K) \ne \varnothing$.
\end{theorem}

If the system has zero deficiency, then the result is also a special case of the Corollary. The argument in this case is largely identical to that in the previous example, though the justifying results are different. Due to deficiency zero, the linkage class decomposition is both independent and ZDD. Also since $(\mathscr{L}_i, K_i)$ has zero deficiency, it follows from a 1972 result of Feinberg that each is absolutely complex balanced. Then it is also of PLP type with $P_E = \widetilde{S}$ according to a result of M\"uller and Regensburger. It follows from the definition of PL-TIK that it is factor span surjective on each linkage class.

\section{Summary, conclusions, and outlook}
\label{sec:sum}
We summarize our results and provide some insights for further research works.
\begin{enumerate}
	\item We provided conditions for existence of weakly reversible CF-decompositions of a chemical kinetic system with underlying reaction network in terms of decomposition of linkage classes.
	\item We developed an algorithm for determining the existence of a weakly reversible CF-decomposition of an NF system, which generalize the idea of decomposition with a single NF node. The results are valid for networks with finite number of NF nodes. In addition, we consider existence of weakly reversible decompositions that are independent.
	\item We established existence and parametrization of equilibria for a class of PL-NDK systems with weakly reversible PL-RDK decompositions. In particular, we extended the Deficiency Zero Theorem of Fortun et al. to PL-NDK systems not only those with zero deficiency but also those with positive deficiency.
	\item One interesting direction is to explore more on properties of NF systems with weakly reversible independent and/or incidence independent CF-decompositions.
	In a broader sense, the theory of decompositions of chemical reaction networks and chemical kinetic systems is a promising research field to study the set of positive equilibria of these structures.
\end{enumerate}

\section*{Acknowledgements}
Bryan S. Hernandez
acknowledges the University of the Philippines for
the PCI Bank Group Professorial Chair Award. This research was partially supported by the Institute for Basic Science in Daejeon, Republic of Korea with funding information IBS-R029-C3.



\appendix
\section{List of abbreviations}
\begin{tabular}{ll}
	\noalign{\smallskip}\hline\noalign{\smallskip}
	Abbreviation& Meaning \\
	\noalign{\smallskip}\hline\noalign{\smallskip}
	CF& complex factorizable\\
	CF-RM& complex factorization by reactant multiples\\
	CLP& complex balanced equilibria log parametrized\\
	CRN& chemical reaction network\\
	GMAS& generalized mass action systems\\
	MAK& mass action kinetics\\
	NF& non-complex factorizable\\
	PL-FSK& power law with factor span surjective kinetics\\
	PL-ISK& power law with interaction span surjective kinetics\\
	PLK& power law kinetics\\
	PL-NDK& power law with non-reactant-determined kinetics\\
	PLP& positive equilibria log-parametrized\\
	PL-RDK& power law with reactant-determined kinetics\\
	PL-RLK& power law with reactant set linear independent kinetics\\
	PL-TIK& power law with $\widehat{T}$-rank maximal kinetics\\
	RIDK& rate contant-interaction map decomposable kinetics\\
	SFRF& species formation rate function\\
	\noalign{\smallskip}\hline
\end{tabular}

\end{document}